\documentclass[12pt]{article}
\usepackage{amsmath}
\usepackage{amsfonts}
\usepackage{amsthm}
\usepackage{amssymb}
\usepackage{hyperref}
\usepackage[showonlyrefs]{mathtools}
\mathtoolsset{showonlyrefs}
\def\C{\mathbb C}
\def\N{\mathbb N}
\def\Z{\mathbb Z}
\def\R{\mathbb R}
\def\O{\mathcal O}
\def\diam{\operatorname{diam}}
\def\dim{\operatorname{dim}}
\newtheorem{theorem}{Theorem}[section]
\newtheorem{lemma}[theorem]{Lemma}

\theoremstyle{remark}

\newtheorem*{ack}{Acknowledgment}
\newtheorem{remark}[theorem]{Remark}
\numberwithin{equation}{section}
\begin{document}
\title{Non-escaping points of Zorich maps }
\author{Walter Bergweiler and Jie Ding}
\date{}
\maketitle
\begin{abstract}
We extend results about the dimension of the radial Julia set of certain exponential functions
to quasiregular Zorich maps in higher dimensions. Our results improve on previous estimates
of the dimension also in the special case of exponential functions.
\end{abstract}

\section{Introduction} \label{intro}
The \emph{Julia set} $J(f)$ of an entire function $f$ is the set where the iterates $f^n$ of $f$
do not form a normal family and the \emph{escaping set} $I(f)$ consists of all points which tend
to infinity under iteration of $f$. These sets play a fundamental role in the iteration 
theory of entire functions. A result of Eremenko~\cite{Eremenko1989} says that $J(f)=\partial I(f)$.
We refer to~\cite{Bergweiler1993,Schleicher2010} for an introduction to the iteration theory of entire 
functions.

We consider the exponential family consisting of the functions $E_\lambda(z):=\lambda e^z$
with $\lambda\in\C\setminus\{0\}$.
If $0<\lambda<1/e$, then $E_\lambda$ has an attracting fixed
point. Devaney and Krych~\cite{Devaney1984} showed that then $J(E_\lambda)$ is equal to the
complement of the attracting basin of this fixed point and $J(E_\lambda)$ consists of 
uncountably many pairwise disjoint curves (called \emph{hairs}) which connect a finite point
(called the \emph{endpoint} of the hair) with $\infty$.
Let $C_\lambda$ be the set of endpoints of the hairs that form $J(E_\lambda)$.
The results of Devaney and Krych also yield that $J(E_\lambda)\setminus C_\lambda\subset I(E_\lambda)$.

McMullen~\cite{McMullen1987} showed that $\dim J(E_\lambda)=2$. 
Here and in the following 
$\dim X$ denotes the Hausdorff dimension of a set $X$.
In fact, McMullen showed that
$\dim I(E_\lambda)=2$ and $I(E_\lambda)\subset J(E_\lambda)$.
 Karpi\'nska~\cite{Karpinska1999b} obtained the surprising result that 
$\dim J(E_\lambda)\setminus C_\lambda=1$. 

A $3$-dimensional analogue of the results of Devaney and Krych, McMullen and Karpi\'nska
was obtained in~\cite{Bergweiler2010}. Here the exponential function was replaced by 
a quasiregular map $F\colon \R^3\to\R^3$ introduced by Zorich~\cite[p.~400]{Zorich1967}.
As noted in~\cite[\S~8.1]{Martio1975},
Zorich maps exist in $\R^d$ for all $d\geq 2$,
and this can be used
(see~\cite[Remark~9]{Bergweiler2010} and~\cite{Comduehr2019})
to obtain a $d$-dimensional analogue of the above results.

To define a Zorich map, following~\cite[\S 6.5.4]{Iwaniec2001}, we fix $\rho>0$ and consider the cube
\begin{equation}\label{5a}
Q:=\left\{x\in \R^{d-1}\colon \|x\|_\infty \leq \rho\right\}=[-\rho,\rho]^{d-1}
\end{equation}
and the upper hemisphere
\begin{equation}\label{5b}
U:=\left\{x\in \R^{d}\colon \|x\|_2 = 1, x_d\geq 0 \right\}.
\end{equation}
Let $h\colon Q\to U$ be a bi-Lipschitz map and define
\begin{equation}\label{1a}
F\colon Q\times \R \to \R^d,\
F(x_1,\dots,x_d)=e^{x_d}h(x_1,\dots,x_{d-1}).
\end{equation}
The map $F$ is then extended to a map $F\colon \R^d\to\R^d$ by repeated reflection at hyperplanes.

The main result of~\cite{Bergweiler2010} says that if $a\in\R$ is sufficiently large,
then the map
\begin{equation}\label{1b}
f_a\colon \R^d\to\R^d,\
f_a(x)=F(x)-(0,\dots,0,a),
\end{equation}
has an attracting fixed point $\xi_a$ such that 
the complement of the attracting basin of $\xi_a$ consists of hairs, the set of endpoints of the hairs
has dimension~$d$, but the union of the hairs without the endpoints has dimension~$1$.

The purpose of this paper is to extend some further results about
the exponential family to the higher dimensional setting.
Let $J_{\rm bd}(E_\lambda)$ be the set
of all $z\in J(E_\lambda)$ for which the orbit $\{E_\lambda^n(z)\colon n\in\N\}$ is bounded.
Karpi\'nska~\cite[Theorem~2]{Karpinska1999a} also showed that $\dim J_{\rm bd}(E_\lambda)>1$ for 
all $\lambda\in (0,1/e)$ and 
\begin{equation}\label{5e}
1+\frac{1}{\log\log(1/\lambda)}<\dim  J_{\rm bd}(E_\lambda)<
1+\frac{1}{\log\log\log(1/\lambda)}
\end{equation}
if $\lambda$ is sufficiently small.

Urba\'nski and Zdunik~\cite{Urbanski2003} considered the set
$J_{\rm r}(E_\lambda):=J(E_\lambda)\setminus I(E_\lambda)$.  We note that,
in general, the notation $J_{\rm r}(f)$ is used for the \emph{radial Julia set} of an entire function $f$,
but for the functions $E_\lambda$ with $0<\lambda<1/e$ this agrees with the above definition; 
see~\cite{Rempe2009} for a discussion of radial Julia sets.
Clearly $J_{\rm r}(E_\lambda)\supset J_{\rm bd}(E_\lambda)$.
Urba\'nski and Zdunik proved~\cite[Theorem~4.5]{Urbanski2003} that
$\dim J_{\rm r}(E_\lambda)=\dim J_{\rm bd}(E_\lambda)<2$ for $0<\lambda<1/e$.
They noted that~\eqref{5e} thus yields that
\begin{equation}\label{5e0}
\lim_{\lambda\to 0} \dim J_{\rm r}(E_\lambda)=1,
\end{equation}
but they also gave a direct proof of this~\cite[Theorem~7.2]{Urbanski2003}.

Urba\'nski and Zdunik also showed that the function $\lambda\mapsto \dim J_{\rm r}(E_\lambda)$
is continuous~\cite[Theorem~4.7]{Urbanski2003} in the interval $(0,1/e)$
and in fact real-analytic~\cite[Theorem~9.3]{Urbanski2004a}.
The function $\lambda\mapsto\dim J_{\rm r}(E_\lambda)$, and in particular its behavior as $\lambda\to 1/e$,
was further studied in~\cite{Havard2010,Urbanski2004b}.

We consider the corresponding sets for the Zorich maps. Denoting by $\xi_a$ the attracting fixed 
point of $f_a$ we thus put
\begin{equation}\label{5e1}
J_{\rm bd}(f_a):=\left\{x\in\R^d\colon f_a^n(x)\not\to \xi_a
\ \text{and} \ 
(f_a^n(x))\ \text{is bounded}\right\}
\end{equation}
and
\begin{equation}\label{5e2}
J_{\rm r}(f_a):=\left\{x\in\R^d\colon f_a^n(x)\not\to \xi_a
\ \text{and} \ 
|f_a^n(x)|\not\to \infty\right\}.
\end{equation}
\begin{theorem}\label{thm3}
If $a$ is sufficiently large, then
\begin{equation}\label{5f}
d-1+\frac12 \frac{\log\log a}{\log a} - \frac{\log\log\log a}{\log a}
< \dim J_{\rm bd}(f_a) \leq  \dim J_{\rm r}(f_a) \leq d-1+\frac{\log\log a}{\log a}.
\end{equation}
\end{theorem}
It follows from Theorem~\ref{thm3}  that if $0<\eta<\frac12$ and $a$ is sufficiently large, then
\begin{equation}\label{5f1}
d-1+\eta \frac{\log\log a}{\log a} 
< \dim J_{\rm bd}(f_a),
\end{equation}
but this does not hold for $\eta=1$. It remains open
for which $\eta\in [\frac12,1)$
the inequality~\eqref{5f1} holds.

In order to compare Theorem~\ref{thm3} with~\eqref{5e} we note that for $d=2$, $\rho=\pi/2$,
\begin{equation}\label{5g}
h\colon \left[-\frac{\pi}{2},\frac{\pi}{2}\right]\to \R^2=\C,\  h(x)=(\sin x,\cos x)=\sin x+i\cos x=ie^{-ix},
\end{equation}
and $z=(x,y)=x+iy$ the Zorich map $F$ takes the form
\begin{equation}\label{5h}
F(z)=F(x,y)=e^yh(x) =ie^{y-ix}=ie^{-iz} .
\end{equation}
Hence for $a>0$ and $\lambda=e^{-a}$ we have
\begin{equation}\label{5i}
f_a(z)=F(z)-ia=i(e^{-iz}-a) =(L\circ E_\lambda \circ L^{-1})(z)
\end{equation}
with $L(z)=i(z-a)$.
Thus $f_a$ is conjugate to $E_\lambda$. Since $a=\log(1/\lambda)$ we see that Theorem~\ref{thm3} not only
extends the results for the functions $E_\lambda$ to higher dimensions, but also improves~\eqref{5e}
and~\eqref{5e0} to
\begin{equation}\label{5k}
\begin{aligned}
& \quad \  d-1+\frac12 \frac{\log\log\log(1/\lambda)}{\log\log(1/\lambda)} - \frac{\log\log\log\log(1/\lambda)}{\log\log(1/\lambda)}
\\ &<\dim  J_{\rm bd}(E_\lambda)
\leq \dim  J_{\rm r}(E_\lambda)
\leq 1+\frac{\log\log\log(1/\lambda)}{\log\log(1/\lambda)}.
\end{aligned}
\end{equation}
\begin{ack}
This research was initiated while the second author was visiting the University of Kiel and it was 
completed while the first author was visiting the Shanghai Center of Mathematical Sciences.
Both authors thank the respective institutions for the hospitality.

We also thank the referee 
for a large number of helpful remarks, in particular
for suggesting an improvement of the lower bound in Theorem 1.1.
\end{ack}
\section{Preliminaries} \label{prelims}
We collect some results about Zorich maps which can be found in~\cite{Bergweiler2010,Comduehr2019}.
As mentioned above, in~\cite{Bergweiler2010} only the case $d=3$ is treated, but the 
changes to handle the general case are minor. We also note that in~\cite{Bergweiler2010,Comduehr2019}
only the case $\rho=1$ is considered. However, it was noted already in~\cite[Remark~1]{Bergweiler2010}
that one may replace the unit cube by a cube of other sidelength and in fact by a rectangular box.
The reason that we do not restrict to the case $\rho=1$ is that this way the exponential map is
(conjugate  to) a special Zorich map, as described in the introduction.

We will use $|x|$ for the Euclidean norm of a point $x\in\R^d$; that is, we write $|x|=\|x\|_2$.
For $c\in \R$ we define the half-space
\begin{equation}\label{6a}
H_{> c}:=\left\{(x_1,\dots,x_d)\in \R^d\colon x_d> c\right\}.
\end{equation}
The half-spaces $H_{<c}$, $H_{\geq c}$ and $H_{\leq c}$ and the hyperplane  $H_{=c}$
are defined analogously.

First we note that the derivative $DF(x)$ exists almost everywhere and that if 
$DF(x_1,\dots,x_{d-1},0)$ exists, then
\begin{equation}\label{6b}
DF(x_1,\dots,x_d)=e^{x_d}DF(x_1,\dots,x_{d-1},0).
\end{equation}
This implies that there exist $\alpha,m,M\in \R$ with $0<\alpha<1$ and $m<M$ such that
\begin{equation}\label{6c}
|DF(x)|:=\sup_{|h|=1}|DF(x)(h)|\leq\alpha \quad
\mbox{ a.e. for }x\in H_{\leq m}
\end{equation}
while
\begin{equation}\label{6d}
\ell(DF(x)) :=\inf_{|h|=1}|DF(x)(h)|\geq
\frac{1}{\alpha}\quad \mbox{ a.e. for }x\in H_{\geq M}.
\end{equation}
We may and will assume that $M\geq 0$.
It was shown in~\cite{Bergweiler2010} that if
\begin{equation}\label{6e}
a\geq e^M-m,
\end{equation}
then $f_a$ has an attracting fixed point $\xi_a\in  H_{\leq m}$ and the properties mentioned 
in the introduction hold; that is, the complement of the basin of attraction of $\xi_a$ consists
of hairs, the set of endpoints of the hairs has dimension~$d$, and the union of the hairs without
endpoints has dimension~$1$.

The following result can be considered as an analogue of the result of 
Karpi\'nska~\cite[Theorem~2]{Karpinska1999a} that $\dim J_{\rm bd}(E_\lambda)>1$ for $0<\lambda<1/e$.
\begin{theorem}\label{thm4}
If $a$ satisfies~\eqref{6e}, then $\dim  J_{\rm bd}(f_a)>d-1$.
\end{theorem}
Theorem~\ref{thm4} will be proved in Section~\ref{lowerbound} together with the lower bound in
Theorem~\ref{thm3}.
\begin{remark}
Urba\'nski and Zdunik~\cite[Corollary~7.3]{Urbanski2003} showed that~\eqref{5e0} implies that
$\dim J_{\rm r}(E_\lambda)<2$ whenever $0<\lambda<1/e$ (and in fact whenever $E_\lambda$ has an 
attracting fixed point). The proof uses that if two functions in the exponential family both have an attracting
fixed point, then they are quasiconformally conjugate. This argument is not available in the
higher-dimensional setting.
We do not know whether  $\dim  J_{\rm r}(f_a)<d$ whenever $a$ satisfies~\eqref{6e}.
\end{remark}

For $r=(r_1,\dots,r_{d-1})\in \Z^{d-1}$ we put
\begin{equation}\label{6f}
P(r):= \left\{(x_1,\dots,x_{d-1})\in \R^{d-1}\colon |x_j-2\rho r_j|<\rho \text{ for } 1\leq j\leq d-1\right\}
\end{equation}
so that $P(0)$ is the interior of~$Q$.
Let
\begin{equation}\label{6g}
S:=\left\{ r\in\Z^{d-1}\colon \sum_{j=1}^{d-1} r_j \text{ is even}\right\}.
\end{equation}
Then $F$ maps $P(r)\times \R$ onto $H_{>0}$ if $r\in S$ and onto $H_{<0}$ if $r\in \Z^{d-1}\setminus S$.
Thus $f_a$ maps $P(r)\times \R$ onto $H_{>-a}$ if $r\in S$ and onto $H_{<-a}$ if $r\in \Z^{d-1}\setminus S$.
For $r\in S$ we put
\begin{equation}\label{6h}
T(r) := P(r)\times (M,\infty).
\end{equation}
A short computation (see~\cite[(2.1)]{Bergweiler2010} or~\cite[(2.2)]{Comduehr2019}) shows 
that $f_a(T(r))\supset H_{\geq M}$. Thus there exists a branch 
$\Lambda^r\colon H_{\geq M}\to T(r)$ of the inverse function of~$f_a$.
With $\Lambda:=\Lambda^{(0,\dots,0)}$
we have
\begin{equation}\label{6i}
\Lambda^{(r_1,\dots,r_{d-1})}(x)=\Lambda(x)+(2\rho r_1,\dots,2\rho r_{d-1},0)
\end{equation}
for all $x\in H_{\geq M}$ and all $r\in S$.
We have
\begin{equation} \label{2a}
D\Lambda(x)=Df_a(\Lambda(x))^{-1} =DF(\Lambda(x))^{-1}
\quad \text{ a.e. for } x\in H_{\geq M}.
\end{equation}
It thus follows from~\eqref{6d} that 
\begin{equation} \label{2a1}
|D\Lambda(x)|\leq\alpha
\quad \text{ a.e. for } x\in H_{\geq M}.
\end{equation}
This implies that
\begin{equation}\label{2b}
\left|\Lambda(x)-\Lambda(y)\right|
\leq \alpha |x-y|
\quad\text{for } x,y\in H_{\geq M}.
\end{equation}
Noting that $Df_a(x)=DF(x)$ we deduce from~\eqref{6b} 
that there exist positive constants $c_1$ and~$c_2$ such that
\begin{equation}\label{2c}c_1e^{x_d}\leq\ell\!\left(Df_a(x)\right)\leq|Df_a(x)|\leq c_2e^{x_d}
\quad\text{a.e.}
\end{equation}
It was shown in~\cite{Bergweiler2010,Comduehr2019} that there
exists positive constants $c_3$ and~$c_4$ such that
\begin{equation}\label{2f}
\frac{c_3}{|x|}
\leq\ell\!\left(D\Lambda(x)\right)\leq
\left|D\Lambda(x)\right|\leq\frac{c_4}{|x|}
\quad \text{ a.e. for } x\in H_{\geq M}
\end{equation}
and this was used to prove that
\begin{equation}\label{2h}
\left|\Lambda(x)-\Lambda(y)\right|
\leq c_4\pi\frac{|x-y|}{\min\{|x|,|y|\}} .
\end{equation}

We have to consider how the bounds for $\ell\!\left(D\Lambda(x)\right)$ and $\left|D\Lambda(x)\right|$
in~\eqref{2f} depend on~$a$.
We will write $\overline{a}=(0,\dots,0,a)$  so that $f_a(x)=F(x)-\overline{a}$.
\begin{lemma} \label{lemma1}
There exist constants $c_3$ and $c_4$ depending only on $F$ such that if
$a$ satisfies~\eqref{6e}, then
\begin{equation}\label{2f1}
\frac{c_3}{|x+\overline{a}|}
\leq\ell\!\left(D\Lambda(x)\right)\leq
\left|D\Lambda(x)\right|\leq\frac{c_4}{|x+\overline{a}|}
\quad \text{ a.e. for } x\in H_{\geq M}.
\end{equation}
\end{lemma}
\begin{proof}
By~\eqref{2a}, \eqref{2c}, \eqref{1a} and~\eqref{1b} we have
\begin{equation}\label{2f2}
\begin{aligned}
\left|D\Lambda(x)\right|
&\leq
\frac{1}{\ell\!\left(DF(\Lambda(x))\right)}
\leq \frac{1}{c_1 \exp \!\left(\Lambda_d(x)\right)}
\\ &
= \frac{1}{c_1|F(\Lambda(x))|}
= \frac{1}{c_1|f_a(\Lambda(x))+\overline{a}|}
= \frac{1}{c_1|x+\overline{a}|}.
\end{aligned}
\end{equation}
The proof of the lower bound for $\ell(D\Lambda(x))$ is similar.
\end{proof}
Lemma~\ref{lemma1} implies  that~\eqref{2h} can be improved to
\begin{equation}\label{2h1}
\left|\Lambda(x)-\Lambda(y)\right|
\leq c_4\pi\frac{|x-y|}{\min\{|x+\overline{a}|,|y+\overline{a}|\}}
\end{equation}
for $x,y\in H_{\geq M}$.

\section{Proof of the upper bound in Theorem~\ref{thm3}} \label{upperbound}
For $r\in S$ and $A\subset \R^d$, we will use the notation
\begin{equation}\label{3a1}
\begin{aligned}
A^r
&:=(2\rho r_1,\dots,2\rho r_{d-1},0)+A
\\ &
=\{(2\rho r_1+x_1,\dots,2\rho r_{d-1}+x_{d-1},x_d)\colon x\in A\}.
\end{aligned}
\end{equation}
We also write $B(x,R)$ for the closed ball of radius $R$ around a point $x\in\R^d$.
We note that since $M>m$ it follows from~\eqref{6e} that $a>1$.
\begin{lemma} \label{lemma2}
Let $r,s\in S$ and let $A\subset T(s)$ be bounded. If $a$ satisfies~\eqref{6e}, then
\begin{equation}\label{2h2}
\diam \Lambda(A^r) \leq c_4\pi\frac{\diam A}{\sqrt{\rho^2|r+s|^2+a^2}}.
\end{equation}
\end{lemma}
\begin{proof}
Since $A^r\subset T(r+s)$ we find that if $x=(x_1,\dots,x_d)\in A^r$,
then $|x_j|\geq \max\{2\rho|r_j+s_j|-\rho,0\}\geq \rho|r_j+s_j|$ for $1\leq j\leq d-1$ while
$x_d\geq M$. Thus, recalling that $M\geq 0$, we find that
\begin{equation}\label{3b}
\begin{aligned}
|x+\overline{a}|
&=\sqrt{\sum_{j=1}^{d-1} x_j^2+(x_d+a)^2}
\\ &
\geq
\sqrt{\sum_{j=1}^{d-1} \rho^2(r_j+s_j)^2+a^2}
= \sqrt{\rho^2|r+s|^2+a^2}.
\end{aligned}
\end{equation}
The conclusion now follows from~\eqref{2h1}.
\end{proof}

\begin{lemma} \label{lemma3a}
There exist positive constants $c_5$ and $c_6$ depending only on $d$
such that if $d-1<t\leq d$ and
$N\geq b\geq 3\sqrt{d-1}$, then
\begin{equation}\label{3c1}
c_5 \frac{b^{d-1-t}}{t-d+1}\left( 1-\left(\frac{N}{b}\right)^{d-1-t}\right) 
\leq
\sum_{\substack{r\in S\\ |r|\leq N}} \frac{1}{(|r|^2+b^2)^{t/2}}
\leq
c_6 \frac{b^{d-1-t}}{t-d+1}.
\end{equation}
Moreover,
\begin{equation}\label{3c2}
\sum_{\substack{r\in S\\ |r|\leq N}} \frac{1}{(|r|^2+b^2)^{(d-1)/2}}
\geq
c_5 \log\frac{N}{b} .
\end{equation}
\end{lemma}
\begin{proof}
For $r=(r_1,\dots,r_{d-1})\in\Z^{d-1}$ let $Q_r$ be the cube with vertices at the points
$(2r_1+e_1,\dots,2r_{d-1}+e_{d-1})$, where $e_j\in\{-1,1\}$ for all~$j$.
For $x\in Q_r$ and $b\geq \sqrt{d-1}/7$ we then have
\begin{equation}\label{3d}
|x|^2\leq  \sum_{j=1}^{d-1} (2|r_j|+1)^2 \leq
 \sum_{j=1}^{d-1} (8r_j^2+1)=8|r|^2+d-1 
\leq 8|r|^2+7b^2
\end{equation}
and thus 
$|x|^2+b^2\leq 8|r|^2+8b^2$.
Hence
\begin{equation}\label{3e}
\int_{Q_r} \frac{dx_1\dots dx_{d-1}}{(|x|^2+b^2)^{t/2}}
\geq \int_{Q_r} \frac{dx_1\dots dx_{d-1}}{(8|r|^2+8b^2)^{t/2}}
= \frac{2^{d-1} 8^{-t/2} }{(|r|^2+b^2)^{t/2}}
= \frac{2^{d-1-3t/2} }{(|r|^2+b^2)^{t/2}} .
\end{equation}
For $|r|\leq N$ we have $Q_r\subset B(0,2N+\sqrt{d-1})\subset B(0,2N+d)$ and thus
\begin{equation}\label{3f}
\begin{aligned}
\sum_{\substack{r\in S\\ |r|\leq N}} \frac{1}{( |r|^2+b^2)^{t/2}}
& \leq
2^{3t/2-d+1} \sum_{\substack{r\in S\\ |r|\leq N}} \int_{Q_r} \frac{dx_1\dots dx_{d-1}}{(|x|^2+b^2)^{t/2}}
\\ &
\leq 
2^{3t/2-d+1} \int_{B(0,2N+d)} \frac{dx_1\dots dx_{d-1}}{(|x|^2+b^2)^{t/2}}
\\ &
=2^{3t/2-d+1} \frac{2\pi^{(d-1)/2}}{\Gamma\!\left( \frac{d-1}{2}\right)}
\int_0^{2N+d} \frac{u^{d-2}}{(u^2+b^2)^{t/2}}du.
\end{aligned}
\end{equation}
Furthermore,
\begin{equation}\label{3g}
\begin{aligned}
\int_0^{2N+d} \frac{u^{d-2}}{(u^2+b^2)^{t/2}}du
&=
b^{d-1-t}\int_0^{(2N+d)/b } \frac{v^{d-2}}{(v^2+1)^{t/2}}dv
\\ &
\leq 
b^{d-1-t}\left( 1 + \int_1^{\infty}  v^{d-2-t}dv\right)
\\ &
=
b^{d-1-t}\left( 1 + \frac{1}{t-d+1} \right)
\leq 
2 \frac{b^{d-1-t}}{t-d+1} .
\end{aligned}
\end{equation}
The right inequality in~\eqref{3c1} now follows from~\eqref{3f} and~\eqref{3g}.

To prove the left inequality, we proceed similarly.
For $r\in S$ let $P_r$ be the cube with vertices at the points
$(2r_1+3e_1,\dots,2r_{d-1}+3e_{d-1})$, where $e_j\in\{-1,1\}$ for all~$j$.
Noting that $(2y-3)^2\geq y^2-3$ for $y\in\R$ we find that if 
$x\in P_r$ and $b\geq 3\sqrt{d-1}$, then 
\begin{equation}\label{3d1}
|x|^2\geq  \sum_{j=1}^{d-1} (2|r_j|-3)^2 \geq \sum_{j=1}^{d-1} (|r_j|^2-3) =
|r|^2-3(d-1)
\geq |r|^2-\frac12 b^2
\end{equation}
and thus $|x|^2+b^2\geq \frac12 |r|^2+\frac12 b^2$.
Hence
\begin{equation}\label{3e1}
\int_{P_r} \frac{dx_1\dots dx_{d-1}}{(|x|^2+b^2)^{t/2}}
\leq \int_{P_r} \frac{dx_1\dots dx_{d-1}}{(\frac12|r|^2+\frac12 b^2)^{t/2}}
= \frac{6^{d-1} 2^{t/2} }{(|r|^2+b^2)^{t/2}}.
\end{equation}
For $x\in\R^{d-1}$ we can choose $r\in S$ such that $x\in P_r$.
In fact, we first choose $r\in \Z^{d-1}$ such that $|x_j-2r_j|\leq 1$ for all $j$, and in 
order to achieve that $\sum_{j=1}^{d-1} r_j$ is even we replace $r_1$ by $r_1+1$ if necessary.
For $x\in P_r$ we have $|r_j|\leq (|x_j|+3)/2$ and since $N\geq 3\sqrt{d-1}$ this implies that
\begin{equation}\label{3e2}
\begin{aligned}
|r|^2
&=\sum_{j=1}^{d-1}r_j^2\leq \frac14 \sum_{j=1}^{d-1}(|x_j|+3)^2 \leq \frac14 \sum_{j=1}^{d-1}2(x_j^2+9)
\\ &
=\frac12 |x|^2 +\frac92(d-1)  \leq \frac12 |x|^2+\frac12 N^2.
\end{aligned}
\end{equation}
For $|x|\leq N$ we thus have $|r|\leq N$ so that 
\begin{equation}\label{3e3}
\bigcup_{\substack{r\in S\\ |r|\leq N}} P_r\supset B(0,N).
\end{equation}
Instead of~\eqref{3f} we now obtain
\begin{equation}\label{3f1}
\begin{aligned}
\sum_{\substack{r\in S\\ |r|\leq N}} \frac{1}{( |r|^2+b^2)^{t/2}}
& \geq
6^{1-d} 2^{-t/2} \sum_{\substack{r\in S\\ |r|\leq N}}
\int_{P_r} \frac{dx_1\dots dx_{d-1}}{(|x|^2+b^2)^{t/2}}
\\ &
\geq
6^{1-d} 2^{-t/2} \int_{B(0,N)} \frac{dx_1\dots dx_{d-1}}{(|x|^2+b^2)^{t/2}}
\\ &
=
6^{1-d} 2^{-t/2} \frac{2\pi^{(d-1)/2}}{\Gamma\!\left( \frac{d-1}{2}\right)}
\int_0^{N} \frac{u^{d-2}}{(u^2+b^2)^{t/2}}du
\end{aligned}
\end{equation}
and instead of~\eqref{3g} we have
\begin{equation}\label{3g1}
\begin{aligned}
\int_0^{N} \frac{u^{d-2}}{(u^2+b^2)^{t/2}}du
&=
b^{d-1-t}\int_0^{N/b } \frac{v^{d-2}}{(v^2+1)^{t/2}}dv
\\ &
\geq
b^{d-1-t}2^{-t/2}\int_1^{N/b}  v^{d-2-t}dv
\\ &
=
\frac{b^{d-1-t}}{t-d+1}2^{-t/2}\left( 1-\left(\frac{N}{b}\right)^{d-1-t}\right) .
\end{aligned}
\end{equation}
The left inequality in~\eqref{3c1} now follows from~\eqref{3f1} and~\eqref{3g1}.

Finally, to prove~\eqref{3c2} we only have to note that for $t=d-1$ we obtain
\begin{equation}\label{3g2}
\begin{aligned}
\int_0^{N} \frac{u^{d-2}}{(u^2+b^2)^{(d-1)/2}}du
&=
\int_0^{N/b } \frac{v^{d-2}}{(v^2+1)^{(d-1)/2}}dv
\\ &
\geq 2^{(1-d)/2}\int_1^{N/b}  \frac{dv}{v}
= 2^{(1-d)/2}\log\frac{N}{b}
\end{aligned}
\end{equation}
instead of~\eqref{3g1}.
\end{proof}

\begin{lemma} \label{lemma3}
There exists a constant $c_7$, depending only on $F$, such that if $a$ satisfies~\eqref{6e},
$d-1<t\leq d$, $s\in S$ and $A\subset T(s)$ is bounded, then
\begin{equation}\label{2h3}
\sum_{r\in S}(\diam \Lambda(A^r))^t\leq   c_7\frac{a^{d-1-t}}{t-d+1} (\diam A)^t .
\end{equation}
\end{lemma}
\begin{proof}
Without loss of generality we may assume that $s=0$.
Applying Lemma~\ref{lemma2} we obtain  
\begin{equation}\label{3c}
\begin{aligned}
\sum_{r\in S}(\diam \Lambda(A^r))^t
&\leq
\sum_{r\in S} \frac{(c_4\pi \diam A)^t}{( \rho^2|r|^2+a^2)^{t/2}}
\\ &
\leq
\left(\frac{c_4\pi \diam A}{\rho}\right)^t \sum_{r\in S} \frac{1}{( |r|^2+(a/\rho)^2)^{t/2}}.
\end{aligned}
\end{equation}
We note that the upper bound in~\eqref{3c1} does not depend on~$N$.
Thus we may take the limit as $N\to\infty$ there, which together with~\eqref{3c} yields that
\begin{equation}\label{3c0}
\begin{aligned}
\sum_{r\in S}(\diam \Lambda(A^r))^t 
&\leq
\left(\frac{c_4\pi }{\rho}\right)^t 
\frac{c_6}{t-d+1}\left(\frac{a}{\rho}\right)^{d-1-t}(\diam A)^t 
\\ &
= \frac{c_6 (c_4\pi)^t }{\rho^{d-1}} \frac{a^{d-1-t}}{t-d+1} (\diam A)^t.
\end{aligned}
\end{equation}
The conclusion follows.
\end{proof}
\begin{proof}[Proof of the upper bound in Theorem~\ref{thm3}]
Let $J(f_a):=\{x\colon f_a^n(x)\not\to \xi_a\}$ be the complement of the attracting basin of~$\xi_a$.
For $R>M$ we consider the set
\begin{equation}\label{3h}
K(R):=\left\{x\in J(f_a)\cap B(0,R)\colon \liminf_{k\to\infty}|f_a^k(x)|< R\right\}.
\end{equation}
Let $c_7$ be the constant from Lemma~\ref{lemma3}.
We show that if $a$ is sufficiently large and $d-1<t\leq d$ such that $\tau:=c_7 a^{d-1-t}/(t-d+1)<1$,
then $\dim K(R)\leq t$.
This implies that $\dim J_{\rm r}(f_a)\leq t$, since
$J_{\rm r}(f_a)=\bigcup_{n\in\N} K(R_n)$ for any sequence $(R_n)$ which tends to~$\infty$.
The conclusion follows from this, since for $t=d-1+\log\log a/\log a$ we have 
$a^{d-1-t}/(t-d+1)=1/\log\log a$.

For $s\in S$, $A\subset T(s)$ and $n\in\N$,  let $X_n(A)$ denote the set of all components
of $f^{-n}(A)$ which are contained in $T(0)$.
If $U\in X_{n}(A)$, then $f(U)$ has the form $f(U)=V^r$ for some $V\in X_{n-1}(A)$ and
some $r\in S$. Equivalently, $U=\Lambda(V^r)$. In turn, if $V\in X_{n-1}(A)$ and
$r\in S$, then $\Lambda(V^r)\in X_n(A)$. Together with Lemma~\ref{lemma3} we thus find that
\begin{equation}\label{3i}
\sum_{U\in X_n(A)}(\diam U)^t
=
\sum_{V\in X_{n-1}(A)} \; \sum_{r\in S}(\diam \Lambda(V^r))^t
\leq \tau
\sum_{V\in X_{n-1}(A)} (\diam V)^t.
\end{equation}
Induction yields that
\begin{equation}\label{3j}
\begin{aligned}
\sum_{U\in X_n(A)}(\diam U)^t
&\leq 
\tau^{n-1} \sum_{V\in X_1(A)}(\diam V)^t
\\ &
=
\tau^{n-1} (\diam \Lambda(A))^t
\leq 
\tau^{n-1} (\diam A)^t.
\end{aligned}
\end{equation}
We will apply this for $A^s:=H_{\leq R}\cap T(s)$.
There exists $N\in\N$ such that
\begin{equation}\label{3l}
J(f_a)\cap B(0,R)\subset \bigcup_{\substack{s\in S\\ |s|\leq N}} A^s.
\end{equation}
Next we put
\begin{equation}\label{3k}
Y_n:=\bigcup_{\substack{s\in S\\ |s|\leq N}}\bigcup_{m\geq n} X_m(A^s)
\quad\text{and}\quad
Z_n:=\{ U^s\colon s\in S, |s|\leq N,\; U\in Y_n\}.
\end{equation}
Then $Y_n$ contains all points $x\in T(0)$ for which there exists
$m\geq n$ and $s\in S$ with $|s|\leq N$ such that $f_a^m(x)\in A^s$.
Hence $Z_n$ contains all $x\in \bigcup_{s\in S,|s|\leq N}T(s)$ for which there exists
$m\geq n$ such that $f_a^m(x)\in \bigcup_{s\in S,|s|\leq N}A^s$.
In particular, $Z_n$ contains all $x\in J(f_a)\cap B(0,R)$ for which 
there exists $m\geq n$ such that $|f_a^m(x)|\leq R$.
Thus $K(R)\subset Z_n$ for all $n\in\N$.

Let $L$ be the cardinality of $\{s\in S\colon |s|\leq N\}$. Then
\begin{equation}\label{3m}
\begin{aligned}
\sum_{U\in Z_n}(\diam U)^t
&= L \sum_{U\in Y_n}(\diam U)^t
= L \sum_{\substack{s\in S\\ |s|\leq N}}\sum_{m\geq n} \sum_{U\in X_m(A^s)}(\diam U)^t
\\ &
\leq L \sum_{\substack{s\in S\\ |s|\leq N}}\sum_{m\geq n} \tau^{m-1} \left(\diam A^s\right)^t
= L^2 \sum_{m\geq n} \tau^{m-1} \left(\diam A^0\right)^t
\\ &
= L^2  \left(\diam A^0\right)^t \frac{\tau^{n-1}}{1-\tau}
\end{aligned}
\end{equation}
by~\eqref{3j}.
Since the right hand side tends to $0$ as $n\to\infty$, we deduce that $\dim K(R)\leq t$.
\end{proof}

\section{Proof of the lower bounds} \label{lowerbound}
The proof is based on the theory of iterated functions systems. A similar method was used
in~\cite{Baranski2009} to estimate the dimensions of Julia sets from below.
In particular, we will use the following result~\cite[Proposition~9.7]{Falconer1990}.
\begin{lemma} \label{lemma4}
Let $S_1,\dots,S_m$ be contractions on a closed subset $K$ of $\R^d$ such that
there exists $b_1,\dots,b_m\in (0,1)$ with 
\begin{equation}\label{4b}
b_j|x-y|\leq |S_j(x)-S_j(y)| 
\quad \text{for } x,y\in K \text{ and } 1\leq j\leq m.
\end{equation}
Suppose that $K_0$ is a non-empty compact subset of $K$ with
\begin{equation}\label{4c}
K_0=\bigcup_{j=1}^m S_j(K_0) 
\end{equation}
and $S_j(K_0)\cap S_k(K_0)=\emptyset$ for $j\neq k$.
Let $t>0$ with 
\begin{equation}\label{4d}
\sum_{j=1}^m b_j^t=1. 
\end{equation}
Then $\dim K_0\geq t$.
\end{lemma}
Since the left hand side of~\eqref{4d} is a decreasing function of $t$, it follows
that if
\begin{equation}\label{4e}
\sum_{j=1}^m b_j^t>1,
\end{equation}
then $\dim K_0> t$.
\begin{proof}[Proof of Theorem~\ref{thm4} and the lower bound in Theorem~\ref{thm3}]
Let $N\in\N$ with $N\geq a/\rho$ and put $R:=8\rho N$ and $K:=B(-\overline{a},R)\cap H_{\geq M}$.
Note that if $N$ is sufficiently large, then $R>M-a$ so that $K$ is non-empty.

Next we note that if $x=(x_1,\dots,x_d)\in\R^d$ with $x_d>\log R$, then
\begin{equation}\label{4e0}
|f(x)+a|=|f(x)|=e^{x_d}>R
\end{equation}
and thus $f(x)\notin B(-\overline{a},R)$.
It follows that 
\begin{equation}\label{4e0a}
\Lambda^r(K)\subset A^r:=H_{\leq \log R}\cap T(r)
\end{equation}
for $r\in S$.

Put $L:=a+\log R=a+\log(8\rho N)$.  For $y\in A^r$ we have
\begin{equation}\label{4e1}
|y+\overline{a}|^2\leq \sum_{j=1}^{d-1}(2|r_j|+1)^2\rho^2 +L^2
\leq 8\rho^2|r|^2+(d-1)\rho^2+L^2 \leq 8(\rho^2|r|^2+L^2)
\end{equation}
if $N$ and hence $L$ are large.
Since $a\leq \rho N$ we have $L\leq \rho N+\log (8\rho N)\leq 2\rho N$ if $N$ and $L$ are large.
Thus we see that if $|r|\leq N$ and $y\in A_r$, then
\begin{equation}\label{4e2}
|y+\overline{a}|^2\leq 8( \rho^2 N^2 +4\rho^2 N^2) = 40 \rho^2 N^2 \leq R^2.
\end{equation}
Thus $A^r\subset K$ if $|r|\leq N$, provided $N$ is sufficiently large.
Hence $\Lambda^r(K)\subset K$ if $r\in S$ and $|r|\leq N$.
Together with~\eqref{2a1} this implies that the $\Lambda^r$ are contractions on~$K$.

It follows that the functions $S_j$ of the form $S_j=\Lambda^r\circ \Lambda^s$,
where $r,s\in S$ and $|r|,|s|\leq N$, are also contractions and hence form an
iterated function system on $K$.
We will apply Lemma~\ref{lemma4} to this iterated function system.

It follows from Lemma~\ref{lemma1} that 
$\ell(D\Lambda^r(x))=\ell(D\Lambda(x))\geq c_3/R$ for $x\in K$ and $r\in S$.
By~\eqref{4e1} we also have
\begin{equation}\label{4f}
\ell(D\Lambda^s(y)) 
=\ell(D\Lambda(y)) 
\geq \frac{c_3}{|y+\overline{a}|}
\geq \frac{c_3}{2\sqrt{2}\sqrt{\rho^2|r|^2+L^2}}
\quad \text{for } y\in A^r.
\end{equation}
Hence 
\begin{equation}\label{4g}
\ell(D(\Lambda^s\circ \Lambda^r)(y)) 
\geq
\ell(D\Lambda^s(\Lambda^r(x))) 
\cdot\ell(D\Lambda^r(x)) 
\geq \frac{c_3^2}{2\sqrt{2}R\sqrt{\rho^2|r|^2+L^2}}
\end{equation}
for $x\in K$.
It follows that 
\begin{equation}\label{4h}
|(\Lambda^s\circ \Lambda^r)(x) -(\Lambda^s\circ \Lambda^r)(y)|\geq b_{r,s}|x-y|
\end{equation}
with
\begin{equation}\label{4i}
b_{r,s}=\frac{c_3^2}{2\sqrt{2}R\sqrt{\rho^2|r|^2+L^2}},
\end{equation}
for $|r|,|s|\leq N$ and $x,y\in K$.
The limit set $K_0$ of the iterated function system generated by the 
functions $\Lambda^s\circ \Lambda^r$
is contained in $J_{\rm bd}(f_a)$.
It thus follows from Lemma~\ref{lemma4} and the remark following it that $\dim J_{\rm bd}(f_a)>t$ if 
\begin{equation}\label{4i1}
\sum_{\substack{r\in S\\ |r|\leq N}} \sum_{\substack{s\in S\\ |s|\leq N}} b_{r,s}^{t}>1.
\end{equation}

Let the cube $P_r$ be defined as in the proof of Lemma~\ref{lemma3a}.
Then each $P_r$ has volume $6^{d-1}$ and the union of all $P_r$ for which $|r|\leq N$ covers $B(0,N)$
and thus has volume at least $\pi^{(d-1)/2}N^{d-1}/\Gamma((d+1)/2)$.
Thus the set of all $s\in S$ for which $|s|\leq N$
has at least $\lfloor 6^{1-d}\pi^{(d-1)/2}N^{d-1}/\Gamma((d+1)/2)\rfloor$ elements.
Recalling that $R=8\rho N$ we deduce that there exist a constant $c_8$ such that
with $b:=L/\rho=(a+\log(8\rho N))/\rho$ we have
\begin{equation}\label{4j}
\begin{aligned}
\sum_{\substack{r\in S\\ |r|\leq N}} \sum_{\substack{s\in S\\ |s|\leq N}} b_{r,s}^{t}
&\geq \left\lfloor\frac{6^{1-d}\pi^{(d-1)/2}N^{d-1}}{\Gamma\left(\frac{d+1}{2}\right)}\right\rfloor
\frac{c_3^{2t}}{(2\sqrt{2}R)^{t}}\sum_{\substack{r\in S\\ |r|\leq N}} \frac{1}{(\rho^2|r|^2+L^2)^{t/2}}
\\ &
\geq c_8 N^{d-1-t}\sum_{\substack{r\in S\\ |r|\leq N}} \frac{1}{(|r|^2+b^2)^{t/2}},
\end{aligned}
\end{equation}
for $d-1\leq t\leq d$.

For large $N$ we have $N\geq b=(a+\log(8\rho N))/\rho\geq 3\sqrt{d-1}$. 
Thus Lemma~\ref{lemma3a} is applicable.

Suppose first that $t=d-1$. 
Using~\eqref{3c2} we find with $c_9:=c_8 c_5$ that
\begin{equation}\label{4j1}
\sum_{\substack{r\in S\\ |r|\leq N}} \sum_{\substack{s\in S\\ |s|\leq N}} b_{r,s}^{d-1}
\geq c_9 \log \frac{N}{b}
= c_9 \log \frac{N\rho}{a+\log(8\rho N)}.
\end{equation}
The right hand side of~\eqref{4j1} tends to $\infty$ as $N\to\infty$.
In particular, it is greater than $1$ for large $N$ so that~\eqref{4i1} holds for 
$t=d-1$. Thus $\dim J_{\rm bd}(f_a)>d-1$.
This proves Theorem~\ref{thm4}.

Now we consider the behavior as $a\to\infty$.
Let $t=d-1+\gamma(a)$ where
\begin{equation}\label{4j2}
\gamma(a):=
\frac12 \frac{\log\log a}{\log a} - \frac{\log\log\log a}{\log a}.
\end{equation}
Then $d-1<t\leq d$ for large~$a$.
We also put $\beta(a):=e^{1/\gamma(a)}$ and note that $\beta(a)\to\infty$ as $a\to\infty$.
We may choose $N$ such that $N\sim a\beta(a)$ as $a\to\infty$. 
Since $\log\beta(a)=1/\gamma(a)$ we find that 
$L=a+\log(8\rho N)\sim a$ and hence $b=L/\rho\sim a/\rho$.

With $C:=2/\rho$ we thus have $Nb\leq Ca^{2}\beta(a)$ and $N/b\geq \beta(a)/C$ for large~$a$.
By the definition of $\beta(a)$ we have $\beta(a)^{\gamma(a)}=e$.
For large $a$ we also have $\frac12\leq C^{\gamma(a)}\leq 2$.
It thus follows from~\eqref{4j} and Lemma~\ref{lemma3a} that
\begin{equation}\label{4k}
\begin{aligned}
\sum_{\substack{r\in S\\ |r|\leq N}} \sum_{\substack{s\in S\\ |s|\leq N}} b_{r,s}^{t}
&\geq 
c_9 \frac{(Nb)^{d-1-t}}{t-d+1}\left( 1-\left(\frac{N}{b}\right)^{d-1-t}\right) 
\\ &
\geq
c_9 \frac{\left(Ca^{2}\beta(a)\right)^{-\gamma(a)}}{\gamma(a)}\left( 1-\left(\frac{\beta(a)}{C}\right)^{-\gamma(a)}\right) 
\\ &
\geq
\frac{c_9}{2e} \frac{a^{-2\gamma(a)}}{\gamma(a)}\left( 1-\frac{2}{e}\right)
\end{aligned}
\end{equation}
for large~$a$.
It is easy to see that $\log\gamma(a)=\log\log\log a-\log\log a+\O(1)$ as $a\to\infty$.
It thus follows that 
\begin{equation}\label{4k1}
\log\!\left( \frac{a^{-2\gamma(a)}}{\gamma(a)} \right)
=-2\gamma(a)\log a-\log\gamma(a)
=\log\log\log a +\O(1)
\end{equation}
as $a\to\infty$.
Thus  the right hand side of~\eqref{4k} tends to $\infty$ as $a\to\infty$.
Hence~\eqref{4i1} holds for $t=d-1+\gamma(a)$ and large~$a$. 
This proves the lower bound in Theorem~\ref{thm3}.
\end{proof}

\noindent Mathematisches Seminar\\
Christian-Albrechts-Universit\"at zu Kiel\\
Ludewig-Meyn-Str.\ 4\\
24098 Kiel, Germany

\medskip
\noindent
E-mail: {\tt bergweiler@math.uni-kiel.de}

\bigskip

\noindent School of Mathematics\\
Taiyuan University of Technology\\
Taiyuan 030024\\
China

\medskip
\noindent
E-mail: {\tt dingjie@tyut.edu.cn}
\end{document}